\documentclass[11pt]{amsart}

\newtheorem{theorem}{Theorem}[section]
\newtheorem{lemma}[theorem]{Lemma}
\newtheorem{corollary}[theorem]{Corollary}

\theoremstyle{definition}

\usepackage{tikz-cd}
\usepackage{mathrsfs}
\usepackage[all,cmtip]{xy}

\newcommand{\C}{\mathbb{C}}

\newcommand{\LL}{\mathscr{L}}

\newcommand{\G}{\mathbb{G}}
\newcommand{\D}{\mathbb{D}}

\newcommand{\autd}{Aut (\mathbb{D})}
\newcommand{\autg}{Aut (\mathbb{G})}

\theoremstyle{remark}

\numberwithin{equation}{section}

\begin{document}

\title{On the geometry of the symmetrized bidisc}

\author{Tirthankar Bhattacharyya}
\address{Department of Mathematics, Indian Institute of Science, Bangalore 560012, India.}
\email{tirtha@iisc.ac.in}
\thanks{}

\author{Anindya Biswas}
\address{Department of Mathematics, Indian Institute of Science, Bangalore 560012, India.}
\email{anindyab@iisc.ac.in}
\thanks{}

\author{Anwoy Maitra}
\address{Department of Mathematics, Indian Institute of Science, Bangalore 560012, India.}
\email{anwoymaitra@iisc.ac.in}
\thanks{}



\date{}

\dedicatory{}

\begin{abstract}
	We study the action of the automorphism group of the $2$ complex dimensional manifold symmetrized bidisc $\G$ on itself. The automorphism group is 3 real dimensional. It foliates $\G$ into leaves all of which are 3 real dimensional hypersurfaces except one, viz., the royal variety. This leads us to investigate Isaev's classification of all Kobayashi-hyperbolic 2 complex dimensional  manifolds for which the group of holomorphic automorphisms has real dimension 3 studied by Isaev. Indeed, we produce a biholomorphism between the symmetrized bidisc and the domain
	\[\{(z_1,z_2)\in \mathbb{C} ^2 : 1+|z_1|^2-|z_2|^2>|1+ z_1 ^2 -z_2 ^2|, Im(z_1 (1+\overline{z_2}))>0\}\]
	in Isaev's list. Isaev calls it $\mathcal D_1$. The road to the biholomorphism is paved with various geometric insights about $\G$.
	
	Several consequences of the biholomorphism follow including two new characterizations of the symmetrized bidisc and several new characterizations of $\mathcal D_1$. Among the results on $\mathcal D_1$, of particular interest is the fact that $\mathcal D_1$ is a ``symmetrization''. When we symmetrize (appropriately defined in the context in the last section) either $\Omega_1$ or $\mathcal{D}^{(2)} _1$ (Isaev's notation), we get $\mathcal D_1$.  These two domains $\Omega_1$ and $\mathcal{D}^{(2)} _1$ are in Isaev's list and he mentioned that these are biholomorphic to $\mathbb D \times \D$. We produce explicit biholomorphisms between these domains and $\D \times \D$.

\end{abstract}

\maketitle

\section{Introduction} 

Given a domain $M$ in $\mathbb C^n$, one can ponder the extent to which its group of holomorphic automorphisms $G$ determines $M$ up to biholomorphic equivalence. A remarkable result of Bedford and Dadok \cite{BD} and of Saerens and Zame \cite{SZ} shows that for every compact Lie group $G$, there is a smoothly bounded strongly pseudoconvex domain $M$ in $\mathbb C^N$ for some $N$, for which Aut($M$) $=G$. Moreover, there are uncountably many distinct domains with this property.

This is why it is natural to consider a domain $M$ with non-compact automorphism group. There has been a huge amount of research classifying such domains. Perhaps the best general survey for this topic is Krantz \cite{KrantzSurvey}.

There is a 2 complex dimensional Kobayashi hyperbolic manifold whose automorphism group is non-compact and 3 real dimensional; all orbits, except one, are three real dimensional hypersurfaces and the exceptional orbit is an analytic disc. This is called the symmetrized bidisc:
\begin{align*}
\G =\{(z_1 + z_2, z_1 z_2):z_1,z_2\in \mathbb{D}\},
\end{align*}
where $\mathbb{D}$ denotes the open unit disc in the complex plane $\mathbb{C}$. The above mentioned and other geometric properties of $\mathbb G$ are discussed in Section 2. The properties remind us of one of the classical domains which first appeared in Cartan \cite{Cartan} (page 61), and is greatly studied by Isaev:
\[\mathcal D_1 = \{(z_1,z_2)\in \mathbb{C}^2 : 1+|z_1|^2-|z_2|^2>|1+ z_1 ^2 -z_2 ^2|, Im(z_1 (1+\overline{z_2}))>0\}.\]
The natural question is whether these two domains are biholomorphic.

One of the aims of this paper is to explicitly exhibit a biholomorphic map between the symmetrized bidisc and $\mathcal D_1$. This is where the seminal work \cite{Isaev} is useful where Isaev classified all domains in $\mathbb C^2$ which have $3$ (real) dimensional automorhism groups. We show the biholomorphism in Section 3. This biholomorphic identification with one of Isaev's domains immediately leads to a couple of new characterizations of the symmetrized bidisc. In the final section (Section 4), we give several applications. The applications include exhibiting explicit biholomorphisms between $\D \times \D$ and $\Omega_1$ as well as $\D \times \D$ and $\mathcal{D}^{(2)} _1$. The domains $\Omega_1$ and $\mathcal{D}^{(2)} _1$ are from Isaev's list and he had mentioned that they are biholomorphic to $\D\times \D$ although no explicit formula were known so far. We also show that $\mathcal{D}_1$ is a ``symmetrization" of $\Omega_1$, as well as of $\mathcal{D}^{(2)} _1$ by giving explicit maps.

We would like to mention here that recently a geometric characterization of the symmetrized bidisc was found by Agler, Lykova and Young in \cite{Agler Lykova Young}. While they, roughly speaking, fibrated the symmetrized bidisc over an analytic disc called the royal disc, we fibrate it over an interval of the real line. Consequently, the fibres obtained in \cite{Agler Lykova Young} were themselves analytic discs while in our case the fibres, with the exception of one, are three real dimensional hypersurfaces. This is what finally led us to $\mathcal D_1$.

Very rarely, one finds a domain that is equally interesting to complex analysts and operator theorists. Apart from the long-studied Euclidean ball and the polydisc, the only other domain where operator theory is very rich and complex analysis is highly advanced is the symmetrized bidisc, see \cite{ay-jfa}, \cite{BPSR}, \cite{J-F-Invariant} and \cite{KZ}.
\section{Intrinsic geometry of $\mathbb G$}\label{Geometry of G}

Consider the map
\begin{align}
&sym :\mathbb{D} \times \mathbb{D}  \rightarrow \mathbb{C} \times \mathbb{C}\notag\\
&\underline{z}=(z_1, z_2)  \mapsto (z_1 +z_2, z_1 z_2).\label{Definition of sym}
\end{align}
It is a proper holomorphic map (see \cite{J-F-Invariant}, page 247). Thus, $\G$ is a proper holomorphic image of the bidisc.

Unlike the automorphism groups of the unit polydisc or the unit ball in $\C ^n, n\geq 1$, the automorphism group of $\G$ does not act transitively. This key difference  is the heart of this paper. Let $D= \{(z,z):z\in \mathbb{D}\}$ and $\Delta=\{(2z,z^2):z\in \D \}$. Then $sym(D)=\Delta$ is called the \textit{royal variety} of $\G$. It follows, from the explicit description of $\autg$ provided below that  $\Delta$ is invariant under the action of $Aut (\G)$ and it acts transitively on $\Delta$. In this section, we shall explore some of the properties of the orbits of the action of the automorphism group on $\G$. First, in Subsection 2.1, we shall show that all orbits except the one mentioned above are real three-dimensional hypersurfaces. These three-dimensional orbits give a foliation of $\G - \Delta$. Then, in the next subsection, it will be established that each three-dimensional orbit can be realized as a $\mathbb Z_2$ action on $\autd$ and these orbits are diffeomorphic to each other. The last subsection will provide us with the fact that these orbits are strictly pseudoconvex.

We start by noting that $\G$ is Kobayashi hyperbolic because of the general result that any bounded open set in any finite dimensional complex Euclidean space is Caratheodory hyperbolic and the Agler-Young result in \cite{A-Y} that Caratheodory and Kobayashi distances agree on $\G$. $\G$ is the first known example of a non-convex open set on which these two pseudo-hyperbolic distances agree.

For each $\varphi \in Aut(\mathbb{D})$ (the automorphism group of $\mathbb{D}$), we can define an automorphism $H_\varphi$ of $\mathbb{G}$ by
\begin{align*}
H_\varphi (z_1 + z_2, z_1z_2)= (\varphi (z_1) + \varphi(z_2), \varphi (z_1) \varphi(z_2)).
\end{align*}
An interesting and well-known fact is that these are the only automorphisms of $\G$, i.e., the automorphism group of $\G$ is given by
\begin{align*}
Aut (\G)=\{H_\varphi :\varphi\in Aut(\mathbb{D}) \}.
\end{align*}
This can be found in  \cite{J-F-Invariant}.

\subsection{Foliation of $\G$}

We define a relation $`\sim\textrm'$ on $\G$ by stating that $(s,p)\sim (t,q)$ if and only if there is an $H_\varphi \in \autg$ such that $H_\varphi (s,p)=(t,q)$. Note that for all $a$ in the half-open interval $[0,1)$, the pair $(a,0)$ is a member of $\G$.

\begin{theorem}\label{theorem-0}
	The relation  $`\sim\textrm'$ defined above is an equivalence relation. If the equivalence class of the point $(s,p)$ of $\G$ is denoted by $[(s,p)]$, then the equivalence classes are given by $\{[(a,0)]: a\in [0,1)\}$.
\end{theorem}
\begin{proof}
	Clearly  $`\sim\textrm'$ is an equivalence relation. To find the equivalence classes, consider $(s,p)\in \G$. Then there are $z_1,z_2 \in \D$ such that $(s,p)= (z_1 +z_2, z_1 z_2)$. Let for $\alpha\in \D$, $\varphi_\alpha \in \autd$ be given by
	\[\varphi_\alpha (z)= \frac{\alpha -z}{1 - \overline{\alpha} z}.\]
	So we have
	\[H_{\varphi_{z_1}} (z_1 +z_2, z_1 z_2)= (\varphi_{z_1} (z_2),0).\]
	Using a rotation, we see that there is a $\varphi\in \autd $ such that \[H_\varphi (z_1 +z_2, z_1 z_2)= (|\varphi_{z_1} (z_2)|,0).\]
	Since $0\leq |\varphi_{z_1} (z_2)|<1$ and $\varphi_{z_1} :\D \rightarrow \D$ is onto, each equivalence class must contain an element of the form $(a,0)$ for some $a\in [0,1)$.
	
	To complete the proof, it is sufficient to show that $[(a,0)]= [(b,0)]$ if and only if $a=b$ for $a,b \in [0,1)$.
	Start with the assumption that $[(a,0)]= [(b,0)]$ for some $a,b \in [0,1)$.
	Since $(a,0)\sim (b,0)$, there is a $\varphi \in \autd$
	such that $H_\varphi (a,0)= (b,0)$.
	Now, there is a $\theta \in [0,2\pi]$ and an $\alpha \in \D$ such that \[\varphi (z) = e^{i\theta} \frac{\alpha -z}{1- \overline{\alpha}z}.\]
	So $(\varphi(a)+ \varphi(0),\varphi(a) \varphi(0))=(b,0)$. Thus $(\varphi(a), \varphi(0))$ is either $(b,0)$ or $(0,b)$. If $(\varphi(a), \varphi(0))=(b,0)$, then $-e^{i\theta}a=b$ and if $(\varphi(a), \varphi(0))=(0,b)$, then $e^{i\theta}a=b$. Since both of $a$ and $b$ are non-negative, we deduce that $a=b$. Hence $\{[(a,0)]: a\in [0,1)\}$ is the complete list of the equivalence classes.
\end{proof}
By the definition of $`\sim\textrm'$, $\autg$ acts transitively on the equivalence class $[(a,0)]$ for each $a\in [0,1)$. Also note that the equivalence class $[(0,0)]$ is $\Delta.$
\begin{lemma}\label{proposition-1}
	Each equivalence class $[(a,0)], a\in [0,1),$ is a closed path-connected subset of $\G$.
\end{lemma}
\begin{proof}
	It is sufficient to show that that for any $(s,p)$ in the equivalence class $[(a,0)]$ of the point $(a,0)$, there is a path fron $(a,0)$ to $(s,p)$.
	By definition of the equivalence class $[(a,0)]$, there is a $\varphi \in\autd$ such that $H_\varphi (a,0)= (s,p)$. Every such $\varphi$ is of the form \[ \varphi (z) = e^{i\theta} \frac{\alpha -z}{1- \overline{\alpha}z}\]
	for some $\theta \in [0,2\pi]$ and $\alpha\in \D$. For $t\in [0,1]$, let $\theta_t =t\theta + (1-t)\pi $, $\alpha_t =t \alpha$ and \[\varphi_t (z)=e^{i\theta_t} \frac{\alpha_t -z}{1- \overline{\alpha_t}z}.\]
	Clearly $\varphi_t\in \autd$ for all $t\in [0,1]$, $\varphi_1 (z) =\varphi (z)$ and $\varphi_0 (z) =z$ for all $z\in \mathbb{D}$. Define a map $h:[0,1]\rightarrow \G$ by
	\[ h(t)= H_{\varphi_t}(a,0) = (\varphi_t(a)+\varphi_t(0),\varphi_t(a)\varphi_t(0)).\]
	This is a path with $h(0)= (a,0)$ and $h(1)=H_\varphi (a,0)= (s,p)$. This proves that $[(a,0)]$ is path-connected. To show that the equivalence class $[(a,0)]$ is closed, consider a sequence $\{(s_n, p_n)\}_{n=1} ^\infty $ in $[(a,0)]$ and suppose that it converges to $(s,p)\in \G$.
	For each $n$, there is a $\theta_n \in [0,2\pi]$ and an $\alpha_n \in \D$ such that with $\varphi_n (z) = e^{i\theta_n} \frac{\alpha_n - z}{1- \overline{\alpha_n} z}$, one can write
	\[H_{\varphi_n} (a,0)= (s_n,p_n).\]
	Passing to a subsequence, if necessary, we may assume that $\{\theta_n\}_{n=1} ^\infty $ converges to $\theta$ and  $\{\alpha_n\}_{n=1} ^\infty $ converges to $\alpha$, for some $\theta \in [0,2\pi]$ and $\alpha \in\overline{\D}$. So we have
	\[\Bigg(e^{i\theta}\frac{\alpha - a}{1- \overline{\alpha} a} +e^{i\theta}\alpha, e^{2i\theta}\alpha \frac{\alpha - a}{1- \overline{\alpha} a}\Bigg)=(s,p).\]
	If $|\alpha|=1$, then we get $(s,p)=(2e^{i\theta}\alpha, e^{2i\theta}\alpha^2)$. But this is impossible because $(s,p)\in \G$ satisfies $|s|<2$. So $\alpha\in \D$. Thus, we get $H_\varphi (a,0)=(s,p)$ where $\varphi (z) = e^{i\theta} \frac{\alpha -z}{1- \overline{\alpha}z} $ is in $\autd$. This proves that $(s,p)\in [(a,0)]$. Hence $ [(a,0)]$ is closed.
\end{proof}

A fact worth noting is that given any $(s,p),(t,q)\in [(a,0)], a\in (0,1)$, there are exactly two $\varphi \in \autd$ such that $H_\varphi (t,q)= (s,p)$. When $a=0$, the number of such automorphisms is infinite.

From the previous discussion, it is clear that $\G$ has a complex one-dimensional orbit and uncountably many real three-dimensional orbits.

We identify $\mathbb{C}^2$ with $\mathbb{R}^4$, consider the real $4$-manifolds $\mathbb D \times \D$ and $\G$ and the following diagram
\[
\xymatrix{
	\D \times \D \ar[r]^{sym} \ar[dr]_{f= q\circ sym} & \G \ar[d]^{q} \\
	& [0,1)
}
\]
where
\begin{align}\label{Submersions are Mobius distance}
f(z_1, z_2)=q(z_1 + z_2, z_1 z_2)=|\varphi_{z_1} (z_2)|=\Bigg|\frac{z_1 -z_2}{1-\overline{z_1}z_2}\Bigg|.
\end{align}
Thus, $f$ is the M\"obius distance between the points $z_1$ and $z_2$. Clearly both $f$ and $q$ are surjective $C^\infty$ functions.

Note that $\mathbb{Z}_2 =\{\pm 1\}$ has a natural action on $\D \times \D$ given by
\[(+1)\cdot(z_1, z_2)=(z_1, z_2)\,\,\text{and}\,\,
(-1)\cdot(z_1, z_2)=(z_2, z_1).\]
It is clear that  $\G =   \D \times \D\big/\mathbb{Z}_2$. We collect some obvious facts in a lemma. Proofs are omitted.
\begin{lemma}\label{Lem1}
	With $f,sym$ and $q$ as above, the following are true.
	\begin{enumerate}
		\item The fixed point set of $\mathbb{Z}_2$ is the set $D=\{(z,z):z\in \D\}$.
		\item $f= q\circ sym$.
		\item $sym|_{ \D \times \D - D}$ is a 2-to-1 map and $f(\underline{z})= f((-1)\underline{z})\,$ for all $\underline{z}=(z_1, z_2)\in \D \times \D.$
		\item The action of $\mathbb{Z}_2$ on $\D \times \D -D$ is properly discontinuous.
		\item  $\G$ is a smooth 4-manifold.
	\end{enumerate}
\end{lemma}
We want to show that the $f$ defined above is a submersion of $\D \times \D - D$ into the open interval $(0,1)$. To that end, we need the following computations.
\begin{lemma}\label{Lem2}
	Let $\underline{z}=(z_1,z_2)=(x_1, y_1,x_2,y_2)$ and $f(\underline{z})=f(z_1,z_2)=|\varphi_{z_1} (z_2)|$. Then we have the following.
	\begin{enumerate}
		\item $D_{\underline{w}} f =\big(\partial_{x_1}f,\partial_{y_1} f,\partial_{x_2}f,\partial_{y_2}f\big)(\underline{w})\neq 0$ for all $\underline{w} \in \D\times \D-D$.
		\item $D_{\underline{w}} f : T_{\underline{w}} (\D\times \D-D)\rightarrow T_{f(\underline{w})} (0,1)$ is onto for all  $\underline{w} \in \D\times \D-D$.
	\end{enumerate}
\end{lemma}
\begin{proof}
	
	Since $T_{f(\underline{w})} (0,1)$ is one dimensional, $(i)$ implies $(ii)$. So let us prove $(i)$. For $\underline{z}=(x_1, y_1,x_2,y_2)\in \D \times\D -D$, we have \begin{align}\label{defn-of-f}
	f(x_1, y_1,x_2,y_2)=|\varphi_{z_1} (z_2)|
	=\frac{\sqrt{(x_1 -x_2)^2+(y_1 -y_2)^2}}{\sqrt{(1- x_1 x_2 -y_1 y_2)^2+(x_1 y_2 -x_2 y_1)^2}}.
	\end{align}
	Note that $f(z_1, z_2)=0$ if and only if $z_1 =z_2$. Now differentiating (\ref{defn-of-f}) we get
	\begin{align}
	&\frac{\partial_{x_1}f(x_1, y_1,x_2,y_2) }{f(x_1, y_1,x_2,y_2)}=\Bigg\{\frac{x_1 -x_2}{(x_1 -x_2)^2+(y_1 -y_2)^2}+\frac{x_2(1- x_1 x_2 -y_1 y_2) -y_2 (x_1 y_2 -x_2 y_1)}{(1- x_1 x_2 -y_1 y_2)^2+(x_1 y_2 -x_2 y_1)^2}\Bigg\},\label{differential-of-f_x_1}\\
	&\frac{\partial_{y_1}f(x_1, y_1,x_2,y_2) }{f(x_1, y_1,x_2,y_2)}=\Bigg\{\frac{y_1 -y_2}{(x_1 -x_2)^2+(y_1 -y_2)^2}+\frac{y_2(1- x_1 x_2 -y_1 y_2) +x_2 (x_1 y_2 -x_2 y_1)}{(1- x_1 x_2 -y_1 y_2)^2+(x_1 y_2 -x_2 y_1)^2}\Bigg\},\label{differential-of-f_y_1}\\
	&\frac{\partial_{x_2}f(x_1, y_1,x_2,y_2) }{f(x_1, y_1,x_2,y_2)}=\Bigg\{\frac{-(x_1 -x_2)}{(x_1 -x_2)^2+(y_1 -y_2)^2}+\frac{x_1(1- x_1 x_2 -y_1 y_2) +y_1 (x_1 y_2 -x_2 y_1)}{(1- x_1 x_2 -y_1 y_2)^2+(x_1 y_2 -x_2 y_1)^2}\Bigg\},\label{differential-of-f_x_2}\\
	&\frac{\partial_{y_2}f(x_1, y_1,x_2,y_2) }{f(x_1, y_1,x_2,y_2)}=\Bigg\{\frac{-(y_1 -y_2)}{(x_1 -x_2)^2+(y_1 -y_2)^2}+\frac{y_1(1- x_1 x_2 -y_1 y_2) -x_1 (x_1 y_2 -x_2 y_1)}{(1- x_1 x_2 -y_1 y_2)^2+(x_1 y_2 -x_2 y_1)^2}\Bigg\}.\label{differential-of-f_y_2}
	\end{align}
	We have to show that $D_w f =\big(\partial_{x_1}f,\partial_{y_1}f,\partial_{x_2}f,\partial_{y_2}f\big)(\underline{w})\neq 0$ for all $\underline{w} \in \D\times \D-D$.
	Suppose on the contrary, there is a point $\underline{\alpha}=(c_1, d_1 ,c_2,d_2) \in \D\times \D -D$ such that $D_{\underline{\alpha} }f =0$. Clearly $\underline{\alpha}\neq 0$.
	Let \[f(\underline{\alpha})^2 (1- c_1 c_2 -d_1 d_2)=a\,\, \text{and}\,\, f(\underline{\alpha})^2 (c_1 d_2 -c_2 d_1)=b.\]
	Note that for $\underline{z}=(z_1,z_2)=(x_1, y_1,x_2,y_2)\in \D \times \D -D$  we have the following \begin{align*}
	& \partial_{x_1}f(x_1, y_1,x_2,y_2) =0 \\
	&\textit{or,}\,\,\frac{x_1 -x_2}{(x_1 -x_2)^2+(y_1 -y_2)^2}+\frac{x_2(1- x_1 x_2 -y_1 y_2) -y_2 (x_1 y_2 -x_2 y_1)}{(1- x_1 x_2 -y_1 y_2)^2+(x_1 y_2 -x_2 y_1)^2}\ =0\\
	&\textit{or,}\,\, x_1 + 0\cdot y_1 + \big(f(\underline{z})^2 (1- x_1 x_2 -y_1 y_2)  -1\big)x_2 + \big(-f(\underline{z})^2   (x_1 y_2 -x_2 y_1)\big)y_2=0.
	\end{align*}
	Considering the remaining partial derivatives, we can derive analogous equations. So from the equations, (\ref{differential-of-f_x_1}), (\ref{differential-of-f_y_1}), (\ref{differential-of-f_x_2}) and (\ref{differential-of-f_y_2}), we obtain that $D_{\underline{\alpha} }f =0$ if and only if $(x_1, y_1,x_2,y_2)=(c_1, d_1 ,c_2,d_2)$ is a nontrivial solution of the system of equations
	\begin{align}\label{system-of-equations}
	& x_1 +0\cdot y_1 + (a-1)x_2 +(-b)y_2=0\\
	& 0\cdot x_1 + y_1 + b x_2 + (a-1)y_2=0\\
	&(a-1)x_1 +b y_1 +x_2 + 0\cdot y_2=0\\
	&(-b)x_1 + (a-1)y_1 +0\cdot x_2 + y_2 =0.
	\end{align}
	This gives us that the coefficient matrix of the system of these equations has zero determinant. So we get \[2a = a^2 + b^2.\]
	Now recall that \[f(\underline{\alpha})^2 (1- c_1 c_2 -d_1 d_2)=a\,\, \text{and}\,\, f(\underline{\alpha})^2 (c_1 d_2 -c_2 d_1)=b.\]
	Using these values, we obtain
	\[2=|c_1|^2 +|d_1|^2 + |c_2|^2 + |d_2|^2.\]
	But this is a contradiction, because $(c_1,d_1,c_2,d_2)\in \D \times \D$. Hence
	\[D_{\underline{w}} f =\big(\partial_{x_1}f,\partial_{y_1}f,\partial_{x_2}f,\partial_{y_2}f\big)(\underline{w})\neq 0 \text{ for all } \underline{w} \in \D\times \D-D.\]
	This completes the proof.
\end{proof}
The next lemma gives a geometric structure of $\D \times \D- D$ and from this lemma we derive an analogous result on the symmetrized bidisc.
\begin{lemma}\label{Lem3}
	Let $f:\D \times \D \rightarrow [0,1)$ be given by $f(\underline{z})=f(z_1,z_2)=|\varphi_{z_1} (z_2)|$. Then we have the following.
	\begin{enumerate}
		\item $f|_{\D \times \D -D} : \D \times \D -D \rightarrow (0,1)$ is a submersion.
		
		\item $f$ defines a three dimensional foliation $\mathscr{F}$ of $\D \times \D -D$ where the leaves are $\mathscr{F}_a =f^{-1}\{a\}, a\in (0,1)$.
		
		\item Each leaf of $\mathscr{F} =\{\mathscr{F}_a: a\in (0,1)\}$ is a real 3-manifold.
		
		\item For each leaf $\mathscr{F}_a$ of $\mathscr{F}$, the action of $\mathbb{Z}_2$ induces a free properly discontinuous action on $\mathscr{F}_a$.
	\end{enumerate}
\end{lemma}
\begin{proof}
	
	$(1)$ From Lemma \ref{Lem2} we have $D_{\underline{w}} f =\big(\partial_{x_1}f,\partial_{y_1}f,\partial_{x_2}f,\partial_{y_2}f\big)(\underline{w})\neq 0$ for all $\underline{w} \in \D\times \D-D$. So $f|_{\D \times \D -D} : \D \times \D -D \rightarrow (0,1)$ is a submersion.
	
	$(2)$ Using Example 1 at page 23 in \cite{Camacho}, we find that $f$ defines a three dimensional foliation of $\D\times \D -D$ where the leaves are the connected components of $f^{-1} \{a\}, a\in (0,1)$. Now $f^{-1} \{a\}= \{(\varphi(a), \varphi(0)):\varphi \in \autd\}$. This is clearly path-connected, hence connected.
	
	$(3)$ By Lemma \ref{Lem2}, $D_{\underline{w}} f : T_{\underline{w}} (\D\times \D-D)\rightarrow T_{f(\underline{w})} (0,1)$ is onto for all  $\underline{w} \in \D\times \D-D$ and $f:\D\times \D \rightarrow [0,1)$ is surjective. By the Preimage Theorem in \cite{Guillemin} (page 21), each point of $(0,1)$ is a regular value for $f$ and $f^{-1} \{a\}=\mathscr{F}_a$ is a three dimensional submanifold of $\D\times \D -D$.
	
	$(4)$ Under the $\mathbb{Z}_2$ action, $\mathscr{F}_a$ is invariant because $f(\underline{z})=f((-1)\cdot \underline{z})$. Since the $\mathbb{Z}_2$ action is free on $\D\times \D -D$, the $\mathbb{Z}_2$ action  is free on $\mathscr{F}_a$ as well. Now, $\mathscr{F}_a$ is Hausdorff. So as in Lemma \ref{Lem1}, $\mathbb{Z}_2$ acts properly discontinuously on $\mathscr{F}_a$.
\end{proof}
We have reached the main result of this subsection.
\begin{theorem}\label{theorem-2}
	Consider the map $q: \G \rightarrow [0,1)$ defined by $q(z_1 +z_2, z_1 z_2)= |\varphi_{z_1} (z_2)|$. Then we have the following results.
	\begin{enumerate}
		\item  $q|_{\G -\Delta} : \G -\Delta \rightarrow (0,1)$ is a submersion.
		
		\item $q$ defines a three dimensional foliation $\mathscr{L}$ of $\G -\Delta$, where the leaves are $\mathscr{L}_a= q^{-1}\{a\}, a\in (0,1)$.
		
		\item Each leaf of $\mathscr{L}=\{\mathscr{L}_a: a\in (0,1)\}$ is a real 3-manifold.
		
		\item For each $a\in (0,1)$, $\mathscr{L}_a =\mathscr{F}_a\big/ \mathbb{Z}_2$.
	\end{enumerate}
\end{theorem}
\begin{proof}
	
	$(1)$ We have $q\circ sym=f$ where $f$ is given in Lemma \ref{Lem3} and $sym$ is defined in Section \ref{Geometry of G}. For any $\underline{z}=(z_1,z_2)=(x_1,y_1,x_2,y_2)\in \D\times \D -D$, the determinant of the real Jacobian of $sym$ at $(x_1,y_1,x_2,y_2)$ has the value $|z_1 -z_2|^2$. Since $(z_1,z_2)\in \D\times \D -D$, $z_1 \neq z_2$. So on $\D\times \D -D$, $sym$ is a local diffeomorphism. By chain rule, we have
	\[D_{sym(\underline{z})} q = D_{\underline{z}} f \circ \big(D_{\underline{z}}\,sym\big)^{-1}.\]
	Since both of $D_{\underline{z}} f$ and $D_{\underline{z}}\,sym$ are surjective for $\underline{z}\in \D\times \D -D$, we have $q: \G -\Delta \rightarrow (0,1)$ is a submersion.
	
	$(2)$ Using the submersion $q: \G -\Delta \rightarrow (0,1)$, arguments similar to Lemma \ref{Lem3} give us a three dimensional foliation $\mathscr{L}$ of $\G -\Delta$. The leaves are the connected components of $q^{-1}\{a\}, a\in (0,1)$. Now by Theorem \ref{theorem-0}, $q^{-1}\{a\}= [(a,0)]$ and by Lemma \ref{proposition-1} the equivalence classes $[(a,0)], a\in (0,1)$ are path connected. So each of $q^{-1}\{a\}, a\in (0,1)$ is connected. Hence the leaves are $q^{-1}\{a\}, a\in (0,1)$.
	
	$(3)$ For any $\underline{w}\in \G -\Delta$
	\[D_{\underline{w}} q : T_{\underline{w}} \big(\G -\Delta \big)\rightarrow T_{q(\underline{w})}\big(0,1\big)\]
	is surjective and $q: \G -\Delta \rightarrow (0,1)$ is also surjective. So each point of $(0,1)$ is a regular value for $q$ and $\mathscr{L}_1 = q^{-1} \{a\}$ is a three dimensional submanifold of $\G -\Delta$.
	
	$(4)$ All we need to do is to note that $\mathscr{L}_a = sym({\mathscr{F}_a})$ is a leaf in $\big(\D \times \D -D\big)\big/\mathbb{Z}_2=\G -\Delta$  and hence $\mathscr{L}_a= \mathscr{F}_a\big/\mathbb{Z}_2$. \end{proof}

\subsection{The leaves are diffeomorphic}\label{Property of the leaves}

For $\varphi_1, \varphi_2 \in \autd$ consider the automorphism
\[\Phi_{(\varphi_1, \varphi_2)} (z_1,z_2)=(\varphi_1(z_1), \varphi_2(z_2))\]
of the bidisc. It is well-known, see for example \cite{Rudin  poly}, that
\[Aut(\D \times \D)=\{\Phi_{(\varphi_1, \varphi_2)}, \Phi_{(\varphi_1, \varphi_2)\circ \sigma}: \varphi_1, \varphi_2\in \autd \},\]
where $\sigma :\D\times \D \rightarrow \D\times \D$ sends $(z_1,z_2)$ to $(z_2,z_1)$. The proper closed subgroup $G_\D =\{\Phi_{(\varphi, \varphi)}: \varphi\in \autd \}$ of $Aut(\D \times \D)$ does not act transitively on $\D \times \D$ and hence we consider the  equivalence relation $\sim^\prime$ on $\D \times \D$ by declaring
\begin{align*}
(z_1,z_2)\sim^\prime (w_1,w_2)\,\, \textit{if and only if there is a}\,\, \varphi\in \autd\,\, \textit{such that}\,\, (\varphi(z_1), \varphi(z_2)) = (w_1,w_2).
\end{align*}
Each equivalence class contains exactly one element of the form $(a,0)$ for some $a\in [0,1)$. Recall the map $f:\D \times \D  \rightarrow [0,1)$ from the statement of Lemma \ref{Lem3} defined by $f(z_1, z_2)=|\varphi_{z_1} (z_2)|$. Since $\mathscr F_a = f^{-1}(a)$ for all $a$ in the open interval $(0,1)$, we have, by definition of $f$,
\begin{align*}
\mathscr{F}_a
&= [(a,0)]^\prime\,\,(\textit{the $\sim^\prime$ equivalence class containing $(a,0)$})\\
&=\{\Phi_\varphi (a,0): \varphi\in \autd\} (\textit{denoting $\Phi_{(\varphi,\varphi)}$ by $\Phi_\varphi$ for brevity}).
\end{align*}
We want to point out the fact that given any $a$ in the open interval $(0,1)$ and any $\underline{z}\in \mathscr{F}_a$, there is exactly one $\varphi\in \autd$ such that $\Phi_\varphi(a,0)= \underline{z}$. So we are allowed to define a map $Q_a : f^{-1}\{a\}\rightarrow \autd$ sending $\Phi_\varphi(a,0)$ to $\varphi$. The following lemma shows that $Q_a$ is a diffeomorphism.
\begin{lemma}\label{theorem leaves and autD are diffeo}
	$Q_a : \mathscr{F}_a\rightarrow \autd$ is a diffeomorphism.
\end{lemma}
\begin{proof}
	Clearly $Q_a$ is bijective. All we need to show is that $Q_a$ and $Q^{-1} _a$ are smooth. To do that, first let us define two atlases on $\autd$ and $\mathscr{F}_a$.
	
	Let $U_1 =\{\varphi_{\theta, \alpha}: \theta \in (-\pi, \pi), \alpha \in \D\}$ and $U_2 =\{\varphi_{\theta, \alpha}: \theta \in (0, 2\pi), \alpha \in \D\}$, where $\varphi_{\theta, \alpha} (z)= e^{i \theta}\frac{z-\alpha}{1- \overline{\alpha}z}$. We consider two maps $\psi_1 : U_1 \rightarrow (-\pi, \pi)\times \D$ and $\psi_2 : U_2 \rightarrow (0, 2\pi)\times \D$, defined by $\psi_j (\varphi_{\theta, \alpha})=(\theta, \alpha), j=1,2$. Then $\mathscr{A}=\Big\{\Big(U_1,\psi_1\Big),\Big(U_2,\psi_2\Big)\Big\}$ is a smooth atlas on $\autd$ giving it a structure of a smooth real 3-manifold (see \cite{Agler Lykova Young}).
	
	Now consider $V_j = \{\Phi_\varphi (a,0): \phi \in U_j\},j=1,2$. Also define $\mu_1 : V_1 \rightarrow (-\pi, \pi)\times \D$ and $\mu_2 : V_2 \rightarrow (0, 2\pi)\times \D$ by $\mu_j (\Phi_\varphi ) =\psi_j (\varphi), j=1,2$. Note that $Q_a (V_j)=U_j,j=1,2$. Clearly $\mathscr{A}^\prime=\Big\{\Big(V_1,\mu_1\Big),\Big(V_2,\mu_2\Big)\Big\}$ makes $\mathscr{F}_a$ a smooth real 3-manifold. Now a little computation gives us $\psi_j\circ Q_a \circ \mu^{-1} _i$ and $\mu_i \circ Q_a ^{-1}  \circ \psi^{-1} _j\,\, (i,j=1,2)$ are smooth functions. Hence $Q_a$ is a diffeomorphism.
\end{proof}

For every fixed $a$ in the interval $(0,1)$, we can consider the quotient map
\[sym_{\mathscr{F}_a} :\mathscr{F}_a\rightarrow \mathscr{F}_a\big/\mathbb{Z}_2(=\mathscr{L}_a \text{ by Theorem \ref{theorem-2}}).\]
In other words, $sym_{\mathscr{F}_a}=sym|_{\mathscr{F}_a}$. For every fixed $a$ in $(0,1)$, there is also an action of $\mathbb{Z}_2$ on $\autd$ defined by
\begin{align*}
(+ 1)\cdot \varphi =\varphi \text{ and } (-1)\cdot \varphi =\varphi \circ \varphi_a
\end{align*}
where \[\varphi_a (z)= \frac{a-z}{1 -\overline{a }z}= \varphi_{\pi, a}.\]
This free and and properly discontinuous action leads to the quotient manifold $\autd\big/_a \mathbb{Z}_2$ where we have retained the symbol $a$ to emphasize the significance of the number $a$ in $(0,1)$. Now take the quotient map
\begin{align*}
sym_{\autd} : \autd \rightarrow \autd\big/_a \mathbb{Z}_2.
\end{align*}

By the quotient manifold theorem (Theorem 21.10 in \cite{Lee Smooth}), $sym_{\mathscr{F}_a}$ and $sym_{\autd}$ are smooth submersions. With this in our hand, we state our next result.

\begin{lemma}\label{Diffeo b/w quotients}
	For each $a\in (0,1)$ there is a diffeomorphism $\tilde{Q}_a :\mathscr{L}_a \rightarrow \autd\big/_a \mathbb{Z}_2.$
\end{lemma}
\begin{proof}
	Observe that $sym_{\autd}\circ Q_a$ is constant on the fibers of $sym_{\mathscr{F}_a}$ which are precisely $\{\Phi_\varphi (a,0), \Phi_{\varphi\circ \varphi_a}(a,0)\}$. So Theorem 4.30 in \cite{Lee Smooth} gives us a smooth map $J:\mathscr{L}_a \rightarrow \autd\big/_a \mathbb{Z}_2$ and the following commutative diagram:
	\[
	\xymatrix{
		& \mathscr{F}_a \ar[d]^{sym_{\autd}\circ Q_a} \ar@{->}[dl]_-{sym_{\mathscr{F}_a}} \\
		\mathscr{L}_a \ar[r]^{J\,\,\,\,\,\,\,\,\,\,\,\,\,} &\autd\big/_a \mathbb{Z}_2,
	}
	\]
	i.e., $J\circ sym_{\mathscr{F}_a} = sym_{\autd}\circ Q_a$.
	
	It is also clear that $sym_{\mathscr{F}_a} \circ Q^{-1}_a$ is constant on the fibers of $sym_{\autd}$ which are $\{\varphi, \varphi\circ \varphi_a\}$. So there is a smooth map $H : \autd\big/_a \mathbb{Z}_2 \rightarrow \mathscr{L}_a$ such that the diagram
	\[
	\xymatrix{
		& \autd \ar[d]^{sym_{\mathscr{F}_a} \circ Q^{-1}_a} \ar@{->}[dl]_-{sym_{\autd}} \\
		\autd\big/_a \mathbb{Z}_2 \ar[r]^{\,\,\,\,\,\,\,\,\,H} &\mathscr{L}_a
	}
	\]
	is commutative, i.e., $sym_{\mathscr{F}_a} \circ Q^{-1}_a = H\circ sym_{\autd}$.
	It is easy to see that $H=J^{-1}$. If we write $J= \tilde{Q}_a$, then it is our required diffeomorphism. This completes the proof.
\end{proof}

We know that the orbits of the action of the automorphism group on the symmetrized bidisc are given by the collection $\{\mathscr{L}_a: a\in [0,1)\}$ where $\mathscr{L}_0 = \Delta$. The indexing set $[0,1)$ corresponds to the line $\{(a,0): a\in [0,1)\}$ in $\G$. For each $a\in (0,1)$, $(a,0)$ is fixed by $H_{\varphi_0}$ ($\varphi_0$ is the identity map) and $H_{\varphi_a}$. So the collection of the automorphisms fixing the elements of $\{(a,0): a\in (0,1)\}$ varies with $a$. Now we shall exhibit an indexing set which is easier to deal with. We start with the following Lemma.
\begin{lemma}\label{Different Indexing set}
	For each $a\in (0,1)$, there is a unique $b\in (0,1)$ such that $[(a,0)]=[(0,-b^2)]$. Moreover, the map sending $a$ to $b$ is a diffeomorphism of $(0,1)$.	
\end{lemma}
\begin{proof}
	Define $h: (0,1)\rightarrow (0,1)$ by
	\begin{align*}
	h(a)= \frac{a}{1+ \sqrt{1 -a^2}}.
	\end{align*}
	It is clearly invertible and is a diffeomorphism. Now for $b=\frac{a}{1+ \sqrt{1 -a^2}}$ we have
	\begin{align*}
	\frac{b-a}{1- ab} =-b.
	\end{align*}
	So the automorphism $\varphi_b (z)=\frac{b-z}{1- zb}$ of $\D$ sends $0$ to $b$ and $a$ to $-b$. Hence $H_{\varphi_b} (a,0)= (0,-b^2)$. So $[(a,0)]=[(0,-b^2)]$. Now it is easy to see that for $a\in (0,1)$, if $(0,-c^2)\in [(a,0)]$ for some $c\in (0,1)$, then $c=\frac{a}{1+ \sqrt{1 -a^2}}$.\\
	This completes the proof.
\end{proof}

Thus, the collection of the orbits can be written as $\{[(0,-b^2)]: b\in [0,1)\}$. An interesting fact is that for any $b\in (0,1)$, we have $\{H_\varphi \in \autg: H_\varphi (0,-b^2)=(0,-b^2) \}= \{ H_{\varphi_0}, H_{-\varphi_0}\}$, where $\varphi_0$ is the identity in $\autd$. Consider the function $f_1 : \D \times \D \rightarrow [0,1)$, given by
\[f_1 (z_1, z_2)=\frac{|\varphi_{z_1}(z_2)|}{1+ \sqrt{1- |\varphi_{z_1}(z_2)|^2}}.\]
By Lemmae \ref{Lem3} and \ref{Different Indexing set}, $f_1$ gives a foliation of $\D \times \D -D$. For $a\in [0,1)$ and $b=  a/(1+ \sqrt{1 -a^2})$, we have
\[ f_1 ^{-1} \{b\} = f ^{-1} \{a\} = [(a,0)]^\prime = [(-b,b)]^\prime = G_\D \{(-b,b)\}.\]

The discussion so far shows us that for any $a\in (0,1)$, there is a diffeomorphism from $\mathscr{F}_a$ to $\autd$ that sends $\Phi_\varphi(-b,b)$ to $\varphi$ where $b=  a/(1+ \sqrt{1 -a^2})$. With this in our hand, we consider an action of $\mathbb{Z}_2$ on $\autd$ by
\begin{align*}
(+1)\cdot \varphi = \varphi \text{ and } (-1) \cdot \varphi = \varphi \circ (-\varphi_0)
\end{align*}
where $\varphi_0$ is the identity function in $\autd$. This action is free and properly discontinuous. Let us write $\autd\big/_0 \mathbb{Z}_2$ for the quotient space. Clearly, the quotient map $sym_0 : \autd \rightarrow \autd\big/_0 \mathbb{Z}_2$ is a smooth submersion. The same procedure used in the proof of Lemma \ref{Diffeo b/w quotients} gives us that $\LL_a$ is diffeomorphic with $ \autd\big/_0 \mathbb{Z}_2$. As a consequence of this conclusion, we have the following result.
\begin{theorem}
	For any $c \in (0,1)$, $\LL_c$ and $ \autd\big/_c \mathbb{Z}_2$ are diffeomorphic with $ \autd\big/_0 \mathbb{Z}_2$.
\end{theorem}

Later we shall see that if $c$ and $d$ are two distinct points in $(0,1)$, $\LL_c$ and $\LL_d$ are CR-nonequivalent.

\subsection{Pseudoconvexity of the three dimensional orbits}

We end the section with the result which shows that the three dimensional orbits are strongly
pseudoconvex hypersurfaces. In the next section, we shall see that the pseudoconvexity of these orbits will lead us to our main result, namely, the realization of the symmetrized bidisc.

\begin{theorem}\label{Pseudoconvexity of the orbits}
	All the three-dimensional orbits of $\mathbb{G}$ under the action of its automorphism group are strongly pseudoconvex hypersurfaces.
\end{theorem}
\begin{proof}
	We note that all the three-dimensional orbits are $\{ sym(\mathscr{F}_a) : a \in (0,1) \}$. Recall that
	\[
	\mathscr{F}_a = f^{-1}\{a\} = \Big\lbrace (z_1,z_2) \in \mathbb{D}\times \mathbb{D} : |\phi_{z_1}(z_2)|= \frac{|z_1-z_2|}{|1-\overline{z}_2 z_1|}
	= a \Big\rbrace.
	\]
	Therefore, a defining function for the three-dimensional hypersurface $\mathscr{F}_a$ is given by
	\[
	g_a (z_1,z_2)=|z_1-z_2|^2-a^2|1-z_1 \overline{z}_2|^2 : \mathbb{D} \times \mathbb{D} \to \mathbb{R}.
	\]
	Straightforward calculations reveal that
	
	\begin{subequations} \label{E:pds_of_g_a}
		\begin{align*}
		&\partial_{z_1}g_a = (\overline{z}_1-\overline{z}_2) + a^2 \overline{z}_2 (1-\overline{z}_1 z_2); \\
		&\partial_{z_2} g_a = -(\overline{z}_1-\overline{z}_2) + a^2 \overline{z}_1 (1-\overline{z}_2 z_1); \\
		&\partial_{\overline{z}_1 z_1} g_a = 1 - a^2 |z_2|^2; \\
		&\partial_{\overline{z}_2 z_2} g_a  = 1 - a^2 |z_1|^2; \\
		&\partial_{\overline{z}_2 z_1} g_a  = -1 + a^2 (1 - \overline{z}_1 z_2).
		\end{align*}
	\end{subequations}	
	Now note that the complex tangent space to $\mathscr{F}_a$ at an arbitrary point $\underline{z}$ is
	\[
	T^{\C}_{\underline{z}}(\mathscr{F}_a) = \{ (w_1,w_2) \in \C^2 : (\partial_{z_1} g_a)(\underline{z}) w_1 + (\partial_{z_2} g_a)
	(\underline{z}) w_2 = 0 \}.
	\]
	Let $u^{(a)}_{\underline{z}} = -(\partial_{z_2} g_a)(\underline{z})/(\partial_{z_1} g_a)(\underline{z})$ and let
	$v^{(a)}_{\underline{z}} = ( u^{(a)}_{\underline{z}},1 )^T$. Then
	$T^{\C}_{\underline{z}}(\mathscr{F}_a) = \{ \lambda v^{(a)}_{\underline{z}} : \lambda \in \C \}$. Set
	\[
	B^{(a)}_{\underline{z}} =
	\begin{bmatrix}
	(\partial_{z_1 \overline{z}_1}g_a)(\underline{z}) & (\partial_{z_2 \overline{z}_1}g_a)(\underline{z}) \\
	(\partial_{z_1 \overline{z}_2}g_a)(\underline{z}) & (\partial_{z_2 \overline{z}_2}g_a)(\underline{z})
	\end{bmatrix},
	\]
	the Levi matrix of $g_a$. We want to show that $\langle B^{(a)}_{\underline{z}}v,v \rangle > 0$ for every $v \in
	T^{\C}_{\underline{z}}(\mathscr{F}_a) \setminus \{0\}$, where $\langle \boldsymbol{\cdot} \, , \boldsymbol{\cdot} \rangle$ denotes the standard Hermitian inner
	product in $\C^2$. To do this, it is sufficient, from the form of $T^{\C}_{\underline{z}}(\mathscr{F}_a)$ mentioned above, to show that
	$\langle B^{(a)}_{\underline{z}} v^{(a)}_{\underline{z}}, v^{(a)}_{\underline{z}} \rangle > 0$.
	Now $(z_1,z_2)=(\varphi(a),\varphi(0))$ for some $\varphi \in \autd$, where $\varphi$ is given, for some $\theta \in \mathbb{R}$ and
	some $\alpha \in \D$, by
	\[
	\varphi(z) = e^{i\theta} \frac{z-\alpha}{1-\overline{\alpha}z} \quad \forall \, z \in \D.
	\]
	So
	\[ 
	z_1 = e^{i\theta} \frac{a-\alpha}{1-\overline{\alpha}a}, \quad z_2 = -e^{i \theta} \alpha.
	\]
	One has
	\begin{equation} \label{E:xpr_z1Minusz2_1Minusz1z2bar}
	z_1-z_2 = e^{i\theta}\frac{a(1-|\alpha|^2)}{1-\overline{\alpha}a}, \quad 1-z_1\overline{z}_2 =
	\frac{1-|\alpha|^2}{1-\overline{\alpha}a}.
	\end{equation}
	Therefore
	\begin{align}
	u^{(a)}_{\underline{z}} = - \frac{(\partial_{z_2}g_a)(\underline{z})}{(\partial_{z_1}g_a)(\underline{z})} &=
	\frac{(\overline{z}_1-\overline{z}_2)-a^2 \overline{z}_1 (1-z_1\overline{z}_2)}{(\overline{z}_1-\overline{z}_2)
		+a^2 \overline{z}_2 (1-\overline{z}_1z_2)} \notag \\
	&=\frac{ e^{-i\theta}a\frac{1-|\alpha|^2}{1-a\alpha} - a^2e^{-i\theta}\frac{a-\overline{\alpha}}{1-a\alpha}
		\frac{1-|\alpha|^2}{1-a\overline{\alpha}} }{ e^{-i\theta}a\frac{1-|\alpha|^2}{1-a\alpha} - a^2e^{-i\theta}\overline{\alpha}
		\frac{1-|\alpha|^2}{1-a\alpha} } \quad (\text{using } \eqref{E:xpr_z1Minusz2_1Minusz1z2bar}) \notag \\
	&=\frac{1 - a \frac{a-\overline{\alpha}}{1-a\overline{\alpha}} }{1 - a \overline{\alpha}} \notag \\
	&= \frac{1-a^2}{(1-a\overline{\alpha})^2}. \label{E:xpr_u_a_z}
	\end{align}
	Now let $D^{(a)}(\underline{z}) =
	\big\langle B^{(a)}_{\underline{z}} v^{(a)}_{\underline{z}} , v^{(a)}_{\underline{z}} \big\rangle$. Then
	\begin{align}
	D^{(a)}(\underline{z}) &= (\partial_{z_1 \overline{z}_1}g_a)(\underline{z})|u^{(a)}_{\underline{z}}|^2 + (\partial_{\overline{z}_1
		z_2}g_a)(\underline{z}) \overline{ u^{(a)}_{\underline{z}} } + (\partial_{z_1 \overline{z}_2}g_a)(\underline{z})
	u^{(a)}_{\underline{z}} + (\partial_{z_2 \overline{z}_2}g_a)(\underline{z}) \notag \\
	&= (1 - |az_2|^2) |u^{(a)}_{\underline{z}}|^2 - (1 - a^2(1 - z_1\overline{z}_2)) \overline{u^{(a)}_{\underline{z}}}
	- (1 - a^2(1 - \overline{z}_1z_2)) u^{(a)}_{\underline{z}} + (1 - |az_1|^2), \label{E:xpr_D}
	\end{align}
	using the expressions for the partial derivatives of $g_a$ computed earlier. Also, using the expressions for $z_1$, $z_2$,
	$1-z_1\overline{z}_2$ and $u^{(a)}_{\underline{z}}$ in terms of $a,\alpha$ and $\theta$, we obtain that $|az_2|^2=|a\alpha|^2$ and
	\[
	1-|az_1|^2 = \frac{1-a^2}{|1-a\overline{\alpha}|^2} (1+a^2-a\overline{\alpha}-a\alpha).
	\]
	Also,
	\begin{align}
	(1 - a^2 (1 - \overline{z}_1z_2))u^{(a)}_{\underline{z}} &= \frac{1-a\alpha-a^2+|a\alpha|^2}{1-a\overline{\alpha}}
	\frac{1-a^2}{|1-a\alpha|^2} \notag \\
	&= \frac{ (1-a\alpha) (1-a\overline{\alpha}) - a(a-\overline{\alpha}) }{ 1-a\overline{\alpha} }
	\frac{1-a^2}{|1-a\alpha|^2} \notag \\
	&= \left( 1 - a\alpha - a \frac{a-\overline{\alpha}}{1-a\overline{\alpha}} \right) \frac{1-a^2}{|1-a\alpha|^2}.
	\label{E:one_minus_a_z1barz2}
	\end{align}
	Hence
	\begin{equation} \label{E:one_minus_a_z1z2bar}
	(1 - a^2 (1 - z_1\overline{z}_2))\overline{u^{(a)}_{\underline{z}}} = \Big( 1 - a\overline{\alpha} -a\frac{a-\alpha}{1-a\alpha}
	\Big) \frac{1-a^2}{|1-a\overline{\alpha}|^2}.
	\end{equation}
	So from \eqref{E:xpr_D}, \eqref{E:one_minus_a_z1barz2} and \eqref{E:one_minus_a_z1z2bar} we get
	\begin{align*}
	D^{(a)}(\underline{z})&= (1-|a\alpha|^2) |u^{(a)}_{\underline{z}}|^2 - \Big( 1-a\alpha-a
	\frac{a-\overline{\alpha}}{1-a\overline{\alpha}} \Big) |u^{(a)}_{\underline{z}}|\\ & - \Big( 1-a\overline{\alpha}-a
	\frac{a-\alpha}{1-a\alpha} \Big) |u^{(a)}_{\underline{z}}| + |u^{(a)}_{\underline{z}}|(1+a^2-a\overline{\alpha}-a\alpha).
	\end{align*}
	Therefore, first dividing the above equation throughout by $|u^{(a)}_{\underline{z}}|$ and then substituting the known expression for
	$u^{(a)}_{\underline{z}}$ into the resulting right hand side, we get, after some computations,
	\begin{align}
	\frac{D^{(a)}(\underline{z})}{|u^{(a)}_{\underline{z}}|} &= (1-|a\alpha|^2) \frac{1-a^2}{|1-a\overline{\alpha}|^2} - (1-a^2)
	+ a \left( \frac{a-\overline{\alpha}}{1-a\overline{\alpha}} + \frac{a-\alpha}{1-a\alpha} \right) \notag \\
	&= \frac{1-a^2}{|1-a\overline{\alpha}|^2} (a\overline{\alpha}+a\alpha-2|a\alpha|^2) + \frac{a}{|1-a\overline{\alpha}|^2} (2a
	-a^2(\alpha+\overline{\alpha})-(\alpha+\overline{\alpha})+2a|\alpha|^2).
	\end{align}
	Therefore
	\begin{align}
	\frac{|1-a\overline{\alpha}|^2 D^{(a)}(\underline{z})}{|u^{(a)}_{\underline{z}}|} &=(1-a^2)
	(a\overline{\alpha}+a\alpha-2|a\alpha|^2) + a (2a - a^2(\alpha+\overline{\alpha}) - (\alpha+\overline{\alpha}) + 2a|\alpha|^2 )
	\notag \\
	&= 2a^2(1-a\alpha-a\overline{\alpha}+|a\alpha|^2) = 2a^2|1-a\alpha|^2.
	\end{align}
	Hence $D^{(a)}(\underline{z})=2a^2|u^{(a)}_{\underline{z}}| > 0$ (recall that $|a\alpha|<1$, which allows us to cancel
	$|1-a\alpha|^2$ from both sides), so that, 
	by our previous remarks, we can conclude that the real hypersurface $\mathscr{F}_a$ is strongly pseudoconvex.
	Now recall that $sym: \D \times \D - D \to \G -
	\Delta$ is a local biholomorphism (in fact, a 2-sheeted holomorphic covering map), and that it is a surjection from the hypersurface
	$\mathscr{F}_a$ to the hypersurface $\LL_a$. Therefore, by the biholomorphic invariance of the Levi form, it follows that
	$\LL_a$ is also strongly pseudoconvex, as required.
\end{proof}

In this section, we saw that the action of the automorphism group on the symmetrized bidisc foliates it into strongly pseudoconvex three dimensional hypersurfaces with one exception. A search for domains with these properties brings to the fore a classical domain first studied by Cartan \cite{Cartan} and elaborated in the next section.

\section{Biholomorphism between $\G$ and $\mathcal D_1$}\label{D_1 and G}

The geometry of the symmetrized bidisc studied so far shows that it is a 2-dimensional Kobayashi-hyperbolic complex manifold with 3-dimensional automorphism group whose properly discontinuous action foliates $\G$ into orbits all, except one, of which are 3-dimensional strongly pseudoconvex hypersurfaces with the exceptional one being a complex curve. This brings us to Isaev's classification in \cite{Isaev} of all connected 2-dimensional Kobayashi-hyperbolic complex manifolds having 3-dimensional automorphism groups. Amongst the model spaces introduced there are $\mathcal{D}_{s,t}$ and $\mathcal{D}_s$, the definitions of which we reproduce below:
\[
\mathcal{D}_{s,t}  = \{ (z,w) \in \C^2 : s |1+z^2-w^2| < 1+|z|^2-|w|^2 < t|1+z^2-w^2|, \, Im(z(1+\overline{w})) > 0 \},
\]
where $1 \leq s < t \leq\infty$, with the understanding that if $t=\infty$, then $\mathcal{D}_{s,t}$ does not contain the complex curve
\[
\{ (z,w) \in \C^2 : 1+z^2-w^2=0, \, Im(z(1+\overline{w})) > 0 \}.
\]
Furthermore,
\begin{align*}
\mathcal{D}_s &= \{ (z,w) \in \C^2 : s |1+z^2-w^2| < 1+|z|^2-|w|^2, \, Im(z(1+\overline{w})) > 0 \}\\
&=\{ (z,w) \in \C^2 : |1+z^2-w^2| <\frac{1}{s} (1+|z|^2-|w|^2), \, Im(z(1+\overline{w})) > 0 \}
\end{align*}
where $1 \leq s \leq \infty$. We point out two facts
(see (9) of Section~2 in \cite{Isaev}):
\begin{enumerate}
	\item The automorphism group of each $\mathcal{D}_{s,t}$ and $\mathcal{D}_{s}$ is $SO(2,1)^0$,
	which acts on it in the following way:
	\begin{align*}
		\begin{pmatrix}
			a_{11} &a_{12} &a_{13}\\
			a_{21} &a_{22} &a_{23}\\
			a_{31} &a_{32} &a_{33}
		\end{pmatrix} \cdot (z_1, z_2) = \frac{\begin{pmatrix}
				a_{21} +a_{22}z_1 +a_{23}z_2\\
				a_{31} +a_{32}z_1 + a_{33}z_2
		\end{pmatrix}}{
			a_{11} +a_{12} z_1+ a_{13}z_2}
	\end{align*}
	for $\begin{pmatrix}
	a_{11} &a_{12} &a_{13}\\
	a_{21} &a_{22} &a_{23}\\
	a_{31} &a_{32} &a_{33}
	\end{pmatrix} \in SO(2,1)^0$ and $(z_1, z_2)\in \mathcal{D}_s$ or $\mathcal{D}_{s,t}$.
	
	\item The orbits of the action of $Aut(\mathcal{D}_s)$ on $\mathcal{D}_s$ are the pairwise
	CR-nonequivalent strongly pseudoconvex hypersurfaces $\eta_{c}$, $c \in (0,1/s)$, along with the complex curve $\eta_0$, where
	\begin{align*}
		&\eta_0 = \{(z_1,z_2)\in \mathbb{C}^2 :1+ z_1 ^2 -z_2 ^2=0, Im(z_1 (1+\overline{z_2}))>0 \}\,\,\text{and}\\
		& \eta_c =\{(z_1,z_2)\in \mathbb{C}^2 : |1+ z_1 ^2 -z_2 ^2|=c(1+|z_1|^2-|z_2|^2), Im(z_1 (1+\overline{z_2}))>0\}.
	\end{align*}
\end{enumerate}

These sets $\eta_c$ were first mentioned by E. Cartan \cite{Cartan}.

First, we shall show that $\G$ is biholomorphic with
\begin{align}\label{DefnD1}
\mathcal{D}_1 = \{(z_1,z_2)\in \mathbb{C}^2 : 1+|z_1|^2-|z_2|^2>|1+ z_1 ^2 -z_2 ^2|, Im(z_1 (1+\overline{z_2}))>0\}.
\end{align}
This space is mentioned in \cite{Isaev_n_n2-1} as well, where it is stated that $\mathcal{D}_1$ is contained in the space $\mathcal{H}=\{(z_1,z_2)\in \mathbb{C}^2:Im(z_1)>0, z_2 \notin (-\infty , -1]\cup [1, \infty)\}$. For the sake of completeness and for our future reference, we shall state it as a Lemma and give a quick proof.

\begin{lemma}\label{D_1UpperHPlane}
	$\mathcal{D}_1 \subset \mathcal{H}$. If $(z_1,z_2)\in \mathcal{D}_1$, then $(z_1,0)\in \mathcal{D}_1$.
\end{lemma}
\begin{proof}
	Let $(z_1,z_2)\in \mathcal{D}$. Then we have
	\begin{align}
	&1+|z_1|^2-|z_2|^2>|1+ z_1 ^2 -z_2 ^2| \label{D_1first} \;\; \text{  and }\\
	& Im(z_1 (1+\overline{z_2}))>0.\label{D_1second}
	\end{align}
	Clearly, $1+|z_1|^2>|1+ z_1 ^2|$. Let $z_1 = re^{i\theta}$ and $z_2 =te^{i\phi}$. So from (\ref{D_1first}) and (\ref{D_1second}), we get
	\begin{align}
	& r^2 sin^2 \theta > t^2 sin^2 \phi + r^2 t^2 sin^2 (\phi -\theta)\label{D11}\,\, \text{ and } \\
	& sin\theta + t sin(\phi -\theta)>0 \label{D12}.
	\end{align}
	If $sin\theta \leq 0$, then (\ref{D12}) contradicts (\ref{D11}). So $Im(z_1)>0$ and hence $(z_1,0)\in \mathcal{D}_1$.\\
	Now if $z_2 \in (-\infty ,-1]$, then it contradicts (\ref{D_1second}), and if $z_2 \in [1,\infty)$, then it contradicts (\ref{D_1first}).\\
	This completes the proof.
\end{proof}
We now prove that the symmetrized bidisc is biholomorphically equivalent to the unbounded domain $\mathcal{D}_1$.
\begin{theorem}\label{Biholom B/w spaces}
	$\G$ and $\mathcal{D}_1$ are biholomorphic.
\end{theorem}
\begin{proof}
	To motivate the proof, it is worthwhile considering the complex curve
	\[\eta_0 = \{(z_1,z_2)\in \mathbb{C}^2 :1+ z_1 ^2 -z_2 ^2=0, Im(z_1 (1+\overline{z_2}))>0 \}.\] If $(z_1,z_2)\in \eta_0$, then $z_1$ lies in the upper half plane. We know that $z\mapsto i\frac{1+z}{1-z}$ is a biholomorphic function that maps $\D$ onto the upper half plane. So by Lemma \ref{D_1UpperHPlane} there is a point $p^\prime\in\D$ such that \[z_1 = i\frac{1+p^\prime}{1-p^\prime}\,\,\text{and}\,\,z_2 ^2 =-\frac{4p^\prime}{(1-p^\prime)^2}.\]
	Setting $p^\prime=z^2$ and $z_2 = -i\frac{2z}{1-z^2}$ gives us that the map \[z\mapsto \Big(i\frac{1+z^2}{1-z^2},
	-i\frac{2z}{1-z^2}\Big) \]
	is a biholomorphism from $\D$ onto $\eta_0$.
	
	Motivated by the above, consider the map $F: \G\rightarrow \mathbb{C}^2$ defined by
	\begin{align}\label{DefnOfF}
	F(s,p)=\Big(i \frac{1+p}{1-p}, -i \frac{s}{1-p}\Big).
	\end{align}
	Clearly, this map is injective and holomorphic. For $(s,p)\in \G$, there exist $z_1, z_2 \in \D$ such that $(s,p)=(z_1 +z_2, z_1 z_2)$. So we have
	\begin{align*}
	\frac{|1 + \big(\frac{i+ip}{1-p}\big)^2 - \big(-\frac{is}{1-p}\big)^2|}{1 + |\frac{i+ip}{1-p} |^2 - |-\frac{is}{1-p} |^2}= \frac{|\varphi_{z_1}(z_2)|^2}{2- |\varphi_{z_1}(z_2)|^2}\in [0,1).
	\end{align*}
	Also $Im\big(\frac{i+ip}{1-p} \big(1+\overline{\big(-\frac{is}{1-p}\big)}\big)\big)>0$ if and only if $1> |p|^2 + Im(p\overline{s}+\overline{s})$. Since $Im(p\overline{s}+\overline{s})=Im(p\overline{s}-s)$ and $(s,p)$ satisfies $1> |p|^2 + |p\overline{s}-s|$ (see Theorem 2.1 in \cite{A-Y} or Theorem 7.13 in \cite{J-F-Invariant}), we have that $F$ maps $\G$ into $\mathcal{D}_1$.\\
	To prove surjectivity, take a point $(u,v)\in \mathcal{D}_1$. By Lemma \ref{D_1UpperHPlane}, $Im (u)>0$. So there is a unique $q\in \D$ such that $u= \frac{i+iq}{1-q}$. Choose $t\in \C$ so that $v= -\frac{it}{1-q}$. By (\ref{D_1first}) and (\ref{D_1second}), we have
	\begin{align}\label{ConnectionWithG}
	\frac{|t^2 - 4q|}{2(1+ |q|^2)- |t|^2} \in [0,1)\,\,\text{and}\,\, 1> |q|^2 + Im(\overline{t} + p\overline{t}).
	\end{align}
	Now there is a unique set $\{w_1,w_2\}\subset\C$ satisfying $t= w_1 +w_2$ and $q= w_1 w_2$. Since $|q|<1$, we may assume that $|w_1|<1$. From (\ref{ConnectionWithG}) we get $\frac{|\varphi_{w_1}(w_2)|^2}{2- |\varphi_{w_1}(w_2)|^2}\in [0,1)$. So $\varphi_{w_1}(w_2)\in \D$ and hence $w_2 \in \D$. Thus $(t,q)\in \G$ and $F(t,q)= (u,v)$.\\
	The inverse of $F$ is easy to compute and is clearly holomorphic. So this completes the proof.
\end{proof}
This leads to a characterization theorem. We start by following Isaev and call a connected two-dimensional Kobayashi-hyperbolic complex manifold $M$ having a real three-dimensional group of holomorphic automorphisms $\text{Aut}(M)$ a $(2,3)$-manifold.
\begin{theorem}\label{Main Theorem}
	Suppose $M$ is a $(2,3)$-manifold. Let $G(M)$ be that connected component of the automorphism group of $M$ which contains the identity. Suppose that all the orbits of $M$ under $G(M)$, except only one, are strongly pseudoconvex three-dimensional real hypersurfaces and that the sole remaining orbit is a complex curve. Suppose that
	there exists an $\epsilon_0 > 0$ such that for every $c \in (1-\epsilon_0,1)$, there exists a three-dimensional orbit $O$ such that
	$O$ is CR-equivalent to $\eta_{c}$. Then $M$ is biholomorphic to $\G$.
\end{theorem}
\begin{proof}
	Our theorem follows, with very little effort, from Isaev's work.
	It follows from the proof of \cite[Theorem~5.1]{Isaev} that if $M$ is a $(2,3)$-manifold having an orbit under the action of $G(M)$
	that is a complex curve and also having a strongly pseudoconvex codimension-1 orbit that is CR-equivalent to
	$\eta_{c}$ for some $c \in (0,1)$, then $M$ is biholomorphic to $\mathcal{D}_s$ for some $s \in [1,\infty)$. What we have to do is
	show that $s=1$. Assume, to get a contradiction, that $s > 1$. We choose $a$ so that $(1/s) < a <1$. By hypothesis, $M$
	contains a codimension-1 orbit $O$ that is CR-equivalent to $\eta_{a}$. Also, by assumption, $M$ is biholomorphic to $\mathcal{D}_s$; let
	$f$ be a biholomorphism from $M$ to $\mathcal{D}_s$. We have that $f$ takes $O$ to some codimension-1 orbit in $\mathcal{D}_s$. One must, therefore,
	have $f(O)=\eta_{b}$ for some $b \in (0,1/s)$. In particular, $\eta_{b}$ must be CR-equivalent to $\eta_{a}$; but
	that is a contradiction because $b \neq a$. This shows that $s=1$, and so $M$ is biholomorphic to $\mathcal{D}_1$, which,
	as we have seen, is biholomorphic to $\mathbb{G}$.
\end{proof}

Post facto, the automorphism group of $M$ is connected.

\textit{Remark.} It is also possible to obtain the conclusion of the theorem above by making the following formally weaker hypotheses: $M$ is a
$(2,3)$-manifold that has a codimension-2 orbit under $G(M)$ that is a complex curve and there exists an $\epsilon_0>0$ such that
for every $c \in (1-\epsilon_0,1)$, there exists a codimension-1 orbit that is a strongly pseudoconvex hypersurface and,
furthermore, is CR-equivalent to $\eta_{c}$.

We conclude this section with another characterization of $\mathbb G$. We would like to refer to the many characterizations of $\mathbb G$ which can be found in \cite{A-Y} and \cite{J-F-Invariant}. Here we will give a new condition on a point $(s,p)$ of $\mathbb{C}^2$ so that it belongs to $\G$. It has a resemblance with other known conditions, but it is neither trivial nor identical to any known conditions.
\begin{corollary}\label{condition On G}
	An element $(s,p)$ of $\mathbb C^2$ is in $\G$ if and only if the following conditions hold
	\begin{align*}
	1> |p|^2 + Im(\overline{s}p +\overline{s})\,\, \text{and}\\
	2+2|p|^2 > |s|^2 + |s^2 -4p|.
	\end{align*}
\end{corollary}

\begin{proof}
	From (\ref{ConnectionWithG}) it is clear that $(s,p)\in \G$ if and only if $1> |p|^2 + Im(\overline{s}p +\overline{s})\,\, \text{and}\\
	2+2|p|^2 > |s|^2 + |s^2 -4p|.$
\end{proof}

\section{Applications}

In this section, we give several applications of the ideas developed in sections 2 and 3. The maps $q$ and $F$ play big roles.

\subsection{Application 1: Ideals of $C_0 (\G)$}

Here, we give a complete characterization of $\autg$ invariant closed ideals of $C_0 (\G)$, the algebra of all continuous functions on $\G$ vanishing at infinity. Note that if $X$ is either the open unit ball or the open unit polydisc (see \cite{Turgay} and \cite{Rudin Nagel}), then there is no proper nontrivial $Aut(X)$ invariant closed ideal of $C_0(X)$.

\begin{theorem}\label{theorem-3}
	Each $\autg$ invariant closed ideal of $C_0 (\G)$ can be written as $I(E)$, where $E$ is of the form $q^ {-1} \Lambda$ for some closed subset $\Lambda$ of $[0,1)$.
\end{theorem}
\begin{proof}
	Let $\Lambda$ be a closed subset of $[0,1)$. Then the set $E=q^ {-1} \Lambda$ is closed in $\G$ and it satisfies $H_\varphi (E)=E$ for all $\varphi\in \autd$.
	Consider $I(E)=\{f\in C_0 (\G): f|_E \equiv0\}$. It is a closed ideal of $C_0 (\G)$. Since $H_\varphi (E)=E$ for all $\varphi\in \autd$, we have $f\circ H_\varphi \in I(E)$ whenever $f\in I(E)$ and $\varphi \in \autd$. Thus $I(E)$ is a closed $\autg$ invariant ideal of $C_0(\G)$.
	Conversely, suppose that $I$ is a closed $\autg$ invariant ideal of $C_0(\G)$.
	Let \begin{align*}
	E = \bigcap_{f\in I} f^{-1} \{0\} = \{\underline{w}\in \G: f(\underline{w})=0\,\,\text{for all}\,\,f\in I\}.
	\end{align*}
	Then $I=I(E)$ (see Theorem 1.4.6 in \cite{Kaniuth}). For any $a\in [0,1)$, either $q^{-1}\{a\}\cap E=\emptyset$ or  $q^{-1}\{a\}\subset E$. Indeed, if $q^{-1}\{a\}\cap E\neq \emptyset$, choose, if possible, a $\underline{z}$ in $q^{-1}\{a\}$, which is not in $E$ and a $\underline{w}\in q^{-1}\{a\}\cap E$. Since $\underline{z}, \underline{w}\in q^{-1}\{a\}$, there is a $\varphi \in \autd$ such that $\underline{z}=H_\varphi (\underline{w})$.
	Since $\underline{z}\notin E=  \bigcap_{f\in I} f^{-1} \{0\}$, we can find an $f\in I$ such that $f(\underline{z})\neq 0$. Now $f\circ H_\varphi \in I$ because $I$ is $\autg$ invariant and $\underline{w}\in E$. So $f\circ H_\varphi (\underline{w})=0$. But $\underline{z}=H_\varphi (\underline{w})$ gives us $0\neq f(\underline{z})=f\circ H_\varphi (\underline{w})$. This is a contradiction. Hence our claim follows. Setting
	\[ \Lambda = \{a\in [0,1): q^{-1}\{a\}\cap E \neq \emptyset\} = \{a\in [0,1): q^{-1}\{a\}\subset E \},\]
	it is easy to see that $E= q^{-1}\Lambda$. The only thing that remains to be shown is that $\Lambda$ is closed.
	
	Let $\{a_n\}_{n=1} ^\infty$ be a sequence in $\Lambda$ and it converge to $a\in [0,1)$. Clearly $\{(a_n,0)\}_{n=1} ^\infty$ is a subset of $E$ and it converges to $(a,0)$. Since $E$ is closed, $(a,0)\in E$. Surjectivity of $q$ gives $q(E)=\Lambda $. So $a\in \Lambda$. This shows that $\Lambda$ is closed. The proof is now complete.
\end{proof}

\subsection{Application 2: An exhaustion of $\G$ and new characterizations of $\mathcal D_1$}

An exhaustion of $\mathcal{D}_1$ can be obtained from  \cite{Isaev} by first considering the family of domains
\begin{align} \mathcal{D}_c=\{(z_1,z_2)\in\C^2:1+|z_1|^2 -|z_2|^2>c|1+z_1 ^2 -z_2 ^2|,\,Im(z_1 (1+\overline{z_2}))>0\},\,\,c\geq 1 \label{Definition of D_c}\end{align}
and then noting that
\[\mathcal{D}_1 =\bigcup_{c>1} \mathcal{D}_c. \]
We also have $Aut(\mathcal{D}_c)=Aut(\mathcal{D}_1)=SO(2,1)^0$.
Using the biholomorphism $F:\G \rightarrow\mathcal{D}_1$ from Theorem \ref{Biholom B/w spaces}, we obtain
\[\G=\bigcup_{c>1} F^{-1}(\mathcal{D}_c)=\bigcup_{c>1} \G_c\]
where $\G_c=F^{-1}(\mathcal{D}_c)$ for all $c>1$. Let us find an expression for these $\G_c$'s in terms of $(s,p)$. To begin with, note that for any $c> 1$, $\mathcal{D}_c\subset \mathcal{D}_1$. Hence, an element $(z_1,z_2)\in \mathcal{D}_c$ must have the form
\[ (z_1 , z_2 ) = \Big(i\frac{1+p}{1-p} , -i\frac{s}{1-p}\Big)\]
for some $(s,p)\in \G$. This expression for $(z_1,z_2)$ along with the definition of $\mathcal{D}_c$ in  (\ref{Definition of D_c}), gives us an exhaustion of $\G$.
\begin{theorem}
	For each $c>1$ there is an open set
	\[\G_c =\{(s,p)\in \G : c|s^2 -4p|+ |s|^2 < 2(1+ |p|^2), |p|^2 + Im(\overline{s}p +\overline{s})<1\}\]
	such that $Aut(\G_c)=Aut(\G)\simeq \autd$ and $\G =\bigcup_{c>1} \G_c$.
\end{theorem}
While the definition of $\mathcal{D}_1$ is given by (\ref{Definition of D_c}) and no other characterization is known, there are several defining characterizations of $\G$ known in literature. We collect Agler and Young's results from \cite{A-Y} along with our Corollary \ref{condition On G}.
\begin{theorem} \label{Characterization of G}
	The following statements are equivalent:
	\begin{enumerate}
		\item $(s,p)\in \G$.
		\item $|s^2 -4p|+ |s|^2 < 2(1+ |p|^2)$ and $|p|^2 + Im(\overline{s}p +\overline{s})<1$.
		\item The roots of the equation $z^2 -s z+ p=0$ lie in $\D$.
		\item $|s-\overline{s}p| +|p|^2<1$.
		\item $|s|<2$ and $|s-\overline{s}p| +|p|^2<1$.
		\item $$\Big|\frac{2zp -s}{2-zs}\Big|<1 \text{ for any } z\in \overline{\D}.$$
		\item $$\Big|\frac{2p - \overline{z}s}{2-zs}\Big|<1 \text{ for any } z\in \overline{\D}.$$
		\item $2|s-\overline{s}p| + |s^2 -4p| +|s|^2 <4$.
		\item $|p|<1$ and there exists a $\beta \in \D$ such that $s= \beta p +\overline{\beta}$.
	\end{enumerate}
\end{theorem}
See \cite{J-F-Invariant} for lucid proofs of all of the above except $(2)$. With this in our hand, we are ready to give many defining characterizations of $\mathcal{D}_1$.
\begin{theorem}
	For a point $(u,v)\in \C^2$, the following conditions are equivalent:
	\begin{enumerate}
		\item $(u,v)\in\mathcal{D}_1$.
		\item $|1+u ^2 -v ^2|<1+|u|^2 -|v|^2$ and $0<Im(u (1+\overline{v}))$.
		\item The roots of the equation $(u +i)z^2 +2v z+(u-i)=0$ lie in $\D$.
		\item $|Im(v)+ i\, Im(\overline{u}v)|<Im(u)$.
		\item $|v|<|u+ i|$ and $|Im(v)+ i\, Im(\overline{u}v)|<Im(u)$.
		\item $$\Big|\frac{\alpha(u -i) +v}{u+i +\alpha v}\Big|<1 \text{ for any }\alpha\in \overline{\D}.$$
		\item $$\Big|\frac{u-i +\overline{\alpha}v}{u+i +\alpha v}\Big|<1 \text{ for any } \alpha\in \overline{\D}.$$
		\item $2 |Im(v)+ i\, Im(\overline{u}v)|+ |1+ u ^2 - v ^2|<|i+u|^2 -|v|^2$.
		\item $Im(u)>0$ and there is a point $\beta =\beta_1 +i\beta_2 \in \D\,(\beta_1, \beta_2\in \mathbb{R})$ such that
		\[v + \beta_1 u + \beta_2=0.\]
	\end{enumerate}
\end{theorem}
\begin{proof}
	The proof follows by noting that under the biholomorphism between $\G$ and $\mathcal{D}_1$, the relation between $(s,p)$ in $\G$ and $(u,v)$ in $\mathcal D_1$ is given by
	\[p = \frac{u-i}{u+i}\,\, \text{and}\,\,s=-\frac{2v}{u +i}\]
	and then using Theorem \ref{Characterization of G}. \end{proof}

\subsection{Application 3: The domain $\mathcal{D}_1$ is a symmetrization}

Isaev,  in \cite{Isaev_n_n2-1} and \cite{Isaev}, mentioned two spaces, which he denoted by $\Omega_1$ and $\mathcal{D}^{(2)} _1$. He reasoned that these are biholomorphic to the bidisc $\D \times \D$. The success of the map we have constructed in Theorem \ref{Biholom B/w spaces} is that now we have explicit expressions for the maps from $\D \times \D$ onto $\Omega_1$ and onto $\mathcal{D}^{(2)} _1$. We have also found two proper holomorphic maps from $\Omega_1$ and $\mathcal{D}^{(2)} _1$ onto $\mathcal{D}_1$ which are equivalent to $sym$.

We start by exhibiting a concrete biholomorphic map between $\D \times \D$ and
\begin{align}\Omega_1 =\{(u,v)\in \C^2: |u^2|+|v^2|-1< |u^2 +v^2 -1|\}.\end{align}\label{Omega_1}
The following lemma is crucial to the biholomorphic map that we are going to give.
\begin{lemma}\label{Omega_1 and Slit Plane}
	The set $\{1- u^2 -v^2: (u,v)\in \Omega_1\}$ is contained in the slit plane $\C -(\infty, 0]$.
\end{lemma}
\begin{proof}
	Consider $a\geq 0$ and the equation $1- u^2 -v^2=-a$. If $(u,v)\in \Omega_1$, then we have
	\[1+a= u^2 +v^2=|u^2 +v^2|\leq |u^2| + |v^2|<|u^2 +v^2-1|+1=a+1\]
	This is a contradiction and hence it follows that $\{1- u^2 -v^2: (u,v)\in \Omega_1\}$ and $(\infty,0]$ are disjoint in $\C$.
\end{proof}
We can conclude from this lemma that the map $(z,w)\mapsto \sqrt{1-z^2 -w^2}$ is a holomorphic function from $\Omega_1$ to $\C$. It will help us to find the symmetrization map from $\Omega_1$ onto $\mathcal{D}_1$.
\begin{theorem}\label{Bidisc and Omega_1}
	$\D \times \D$ and $\Omega_1$ are biholomorphic.
\end{theorem}
\begin{proof}
	To give a motivation, set $u^2 +v^2 -1=u_1 ^2$ and $v=v_1$. Then the inequality $|u^2|+|v^2|-1< |u^2 +v^2 -1|$ can be transformed into $|1 + u_1^2 - v_1^2|< 1+ |u_1 ^2| -|v_1|^2$ which is similar to the definition of $\mathcal{D}_1$. So, setting $u_1 = i\frac{1+zw}{1-zw}$ and $v_1 = -i\frac{z+w}{1-zw}$ we obtain
	\[u^2 =\Big(\frac{z-w}{1-zw}\Big)^2\,\,\text{and}\,\, v=-i\frac{z+w}{1-zw}.\]
	At this point, let us define the following map $H: \D \times \D \rightarrow\C^2$ by
	\begin{align}\label{Definition of H}
	H(z,w)=\Big(\frac{z-w}{1-zw} , 	-i\frac{z+w}{1-zw}\Big).
	\end{align}
	Suppose that $H(z,w)=(u,v)$. Then we have
	\[|u ^2 +v ^2 -1|=\Big|\frac{1+zw}{1-zw}\Big|^2\,\, \text{and}\,\, |u ^2| +|v|^2 -1= \frac{|z+\overline{w}|^2 - (1-|z|^2)(1-|w|^2)}{|1-zw|^2}.\]
	A straightforward calculation yields that
	\begin{align}\label{u,v in Omega_1 iff z,w in D}
	|u ^2 +v ^2 -1| > |u ^2| +|v|^2-1\,\,\text{if and only if}\,\,(1-|z|^2)(1-|w|^2)>0.
	\end{align}
	Since $(z,w)\in \D \times \D$, we have that $H$ maps $\D \times \D$ into $\Omega_1$. To show that $H$ is injective, take $(z_1,w_1), (z_2,w_2)\in \D \times \D$ and consider $H(z_1,w_1)=H(z_2,w_2)$. This gives us a system of two equations
	\begin{align*}
	&(z_1 -z_2)+ z_1 z_2 (w_1 -w_2)=0\\
	&w_1 w_2(z_1 -z_2)+  (w_1 -w_2)=0.
	\end{align*}
	Since $(z_1,w_1), (z_2,w_2)\in \D \times \D$, the system of equations in $(z_1 -z_2)$ and $(w_1 -w_2)$ has a unique solution $(0,0)$. Hence $(z_1,w_1)=(z_2,w_2)$. This proves the injectivity of $H$. For the surjectivity of $H$, choose any $(u,v)\in \Omega_1$. There are two cases we need to discuss.\\
	\textit{Case (i):} $u=\pm iv$.\\
	From the definition of $\Omega_1$, we get $|v|<1$ and hence $(iv,0),(0,iv)\in \D \times \D$. Clearly,
	\[H(iv,0)=(iv,v)\,\,\text{and}\,\,H(0,iv)=(-iv,v).\]
	\textit{Case (ii):} $u\neq \pm iv$.\\
	Consider two numbers
	\begin{align*}
	z=\frac{1+ \sqrt{1-u^2 -v^2}}{u-iv}\,\,\text{and}\,\, w=- \frac{1+ \sqrt{1-u^2 -v^2}}{u+iv}.
	\end{align*}
	By Lemma \ref{Omega_1 and Slit Plane}, $z,w$ are well defined elements of $\C$. A little computation gives that $(z,w)$ satisfies the following equations
	\begin{align}\label{u, v in terms of z, w}
	u=\frac{z-w}{1-zw}\,\,\text{and}\,\, v=-i\frac{z+w}{1-zw}.
	\end{align}
	At this point (together with \textit{Case (i)}), we have shown that for any $(u,v)\in \Omega_1$, there are complex numbers $z,w$ such that equations in (\ref{u, v in terms of z, w}) are satisfied. With expressions as in (\ref{u, v in terms of z, w}), and in view of (\ref{u,v in Omega_1 iff z,w in D}), we have \[(1-|z|^2)(1-|w|^2)>0,\]
	since $(u,v)\in \Omega_1$.
	So, if one of $z$ and $w$ lies in $\D$ (or in $\C -\overline{\D}$), so does the other one. Now we conclude the following statements:
	\begin{enumerate}
		\item If $z,w\in \D$, then clearly $H(z,w)=(u,v)$.
		\item If $z,w \in \C -\overline{\D}$, then $-\frac{1}{w},-\frac{1}{z}\in \D$. It is easy to see that $H(-\frac{1}{w},-\frac{1}{z})=(u,v)$.
	\end{enumerate}
	Thus in all cases, there is a pre-image of an arbitrary point of $\Omega_1$ in $\D \times \D$ under the map $H$. So $H$ is surjective.
	Since a bijective holomorphic map always has a holomorphic inverse, we conclude that $H:\D \times \D \rightarrow \Omega_1$ is a biholomorphism (see \cite{Ohsawa}, page 101). This completes the proof.
\end{proof}

Since $\Omega_1$ and and $\mathcal{D}_1$ are biholomorphic with $\D \times \D$ and $\G$, respectively, it is natural to ask for a map from $\Omega_1$ to $\mathcal{D}_1$ which is equivalent to $sym: \D \times \D \rightarrow \G$. The next theorem gives an answer to this question.
\begin{theorem}\label{Equivalent sym map for Omega_1}
	There is a proper holomorphic map $sym_{\Omega_1} : \Omega_1 \rightarrow \mathcal{D}_1$ such that the following diagram is commutative
	\[ \begin{tikzcd}
	\D \times \D \arrow{r}{H} \arrow[swap]{d}{sym} & \Omega_1 \arrow{d}{sym_{{\Omega_1}}} \\%
	\G \arrow{r}{F}& \mathcal{D}_1
	\end{tikzcd},
	\]
	where $F$ and $H$ are defined in Theorem \ref{Biholom B/w spaces} and \ref{Bidisc and Omega_1}, respectively.
\end{theorem}
\begin{proof}
	We start with a point $(u,v)\in \Omega_1$ and its pre-image $(z,w)\in \D \times \D$ under $H$. So,
	\begin{align}\label{u,v in terms of z,w}
	(u,v)=\Big(\frac{z-w}{1-zw} , -i\frac{z+w}{1-zw}\Big).
	\end{align}
	This gives us
	\[\sqrt{1-u^2 -v^2}=\pm \frac{1+zw}{1-zw}.\]
	By Lemma \ref{Omega_1 and Slit Plane}, $\sqrt{1-u^2 -v^2}$ lies in the right half plane. For $(z,w)\in \D \times \D$, the same is true for $\frac{1+zw}{1-zw}$ as well. Hence we conclude
	\[\sqrt{1-u^2 -v^2}= \frac{1+zw}{1-zw}.\]
	Again by  Lemma \ref{Omega_1 and Slit Plane} and the conclusions deduced from it, the map $sym_{\Omega_1}: \Omega_1\rightarrow\mathcal{D}_1$ given by
	\begin{align}\label{Definition of sym_Omega}
	sym_{\Omega_1}(u,v)=\Big(i\sqrt{1-u^2 -v^2},v\Big)
	\end{align}
	is well defined and holomorphic. This is the map that works. Indeed, using equations (\ref{u,v in terms of z,w}) and (\ref{Definition of sym_Omega}), we get
	\[sym_{\Omega_1}\circ H(z,w)=\Big(i\frac{1+zw}{1-zw},-i\frac{z+w}{1-zw}\Big)\]
	which is precisely equals to $F\circ sym(z,w)$ by equations (\ref{Definition of sym}) and (\ref{DefnOfF}).
	
	To see that $sym_{\Omega_1}$ is proper, note that
	\[sym_{\Omega_1} =F\circ sym \circ H^{-1} \]
	and both $F$ and $H$ are biholomorphisms and $sym|_{\D \times \D}$ is a proper map. Since biholomorphisms are homeomorphisms as well, the conclusion that $sym_{\Omega_1} $ is proper, follows. This completes the proof.	
\end{proof}
Let us now consider the remaining biholomorphic copy $\mathcal{D}^{(2)} _1$ of $\D \times \D$. The domain is defined as
\begin{align}\label{Definition of D^2 _1}
\mathcal{D}^{(2)} _1=&\{(1:t:u:v)\in\mathbb{CP}^3:|t|^2+|u|^2-|v|^2>1,t^2 +u^2 -v^2=1,Im(u(\overline{t}+\overline{v})>0)\}\notag\\
&\cup \{(0:t:u:v)\in\mathbb{CP}^3: t^2 +u^2 -v^2=0,Im(u(\overline{t}+\overline{v}))>0\}.
\end{align}
The following theorem provides us with a biholomorphic map from $\D \times \D$ onto $\mathcal{D}^{(2)} _1$.
\begin{theorem}\label{Biholo map between D^2 and {D}^{(2)} _1}The map $J:\D \times \D \rightarrow\mathbb{CP}^3$ defined by
	\begin{align}\label{Definition of J}
	J(z,w)=(z-w:1-zw:i(1+zw):-i(z+w))
	\end{align}
	is a biholomorphism from $\D \times \D$  onto $\mathcal{D}^{(2)} _1$.
\end{theorem}
\begin{proof}
	For $z,w\in \D$, we have $1\pm zw\neq 0$ and hence the map is well defined and is holomorphic. It is easy to see that $J$ maps $\D \times \D$ into $\mathcal{D}^{(2)} _1$.
	
	Let us take two points $(z,w),(z^\prime,w^\prime)\in \D \times \D$ and consider the following equation
	\[(z-w:1-zw:i(1+zw):-i(z+w))=(z^\prime-w^\prime:1-z^\prime w^\prime:i(1+z^\prime w^\prime):-i(z^\prime+w^\prime)).\]
	Equating the third components we get $zw=z^\prime w^\prime$ and then equating the first and fourth components we obtain $(z,w)=(z^\prime,w^\prime)$. So injectivity of $J$ follows. To see the surjectivity, take a point $(s:t:u:v)\in \mathcal{D}^{(2)} _1$. Without loss of generality, we may assume that $s=0\,\,\text{or} \,\,1$. If $s=0$, then by the definition of $\mathcal{D}^{(2)} _1$ (equation (\ref{Definition of D^2 _1})), $(0:t:u:v)$ must satisfy the following equations
	\begin{align}\label{Complex curve in D^2 _1}
	t^2 +u^2 -v^2=0\,\, \text{and}\,\,Im(u(\overline{t}+\overline{v}))>0.
	\end{align}
	Note that $t\neq 0$. For $t=0$, these two equations lead to $Im(u\overline{v})=0$ which is a contradiction. So the equation (\ref{Complex curve in D^2 _1}) can be rewritten as
	\[1+\Big(\frac{u}{t}\Big)^2-\Big(\frac{v}{t}\Big)^2=0\,\,\text{and}\,\,Im\Big(\frac{u}{t}\Big(1+\frac{\overline{v}}{ \overline{t}}\Big)\Big)>0.\]
	This is precisely the equation of the complex curve in $\mathcal{D}_1$ from Section \ref{D_1 and G}. So by Theorem \ref{Biholom B/w spaces}, there is a $z\in \D$ such that
	\[\frac{u}{t}=i\frac{1+z^2}{1-z^2}\,\,\text{and}\,\,\frac{v}{t}=-i\frac{2z}{1-z^2}.\]
	Since $t\neq 0$, it follows that $J(z,z)=(0:t:u:v)$. Now suppose that $s=1$. If $t=0$, then $|u|^2-|v|^2>1$ and $u^2 -v^2=1$ contradicts each other. So $t\neq 0$. Again, appealing to the definition of $\mathcal{D}^{(2)} _1$ and using $t\neq 0$, we conclude that $(1:t:u:v)$ satisfies the following conditions
	\[1+\Big(\frac{u}{t}\Big)^2-\Big(\frac{v}{t}\Big)^2=\frac{1}{t^2},1+\Big|\frac{u}{t}\Big|^2-\Big|\frac{v}{t}\Big|^2>\Big|\frac{1}{t}\Big|^2\,\,\text{and}\,\,Im\Big(\frac{u}{t}\Big(1+\frac{\overline{v}}{ \overline{t}}\Big)\Big)>0.\]
	From these equations, it is clear that $\Big(\frac{u}{t},\frac{v}{t}\Big)\in \mathcal{D}_1$, and hence there exist two distinct elements $z,w\in \D$ such that
	\[\frac{u}{t} = i\frac{1+zw}{1-zw}\,\,\text{and}\,\,\frac{v}{t}=-i \frac{z+w}{1-zw}.\]
	Since $t\neq 0$, we can find a nonzero complex number $\alpha$ such that
	\[\frac{t}{1-zw}=\frac{u}{i(1+zw)}=\frac{v}{-i(z+w)}=\alpha.\]
	Using the relation $t^2 +u^2 -v^2 =1$ we obtain
	\[\alpha =\pm \frac{1}{z-w}.\]
	It is easy to conclude the following statements.
	\begin{enumerate}
		\item If $\alpha=\frac{1}{z-w}$, $J(z,w)=(1:t:u:v)$.
		\item If $\alpha=-\frac{1}{z-w}$, $J(w,z)=(1:t:u:v)$.
	\end{enumerate}
	Hence surjectivity of $J$ follows. This proves that $J$ is a bijective holomorphic mapping from $\D \times \D$ onto $\mathcal{D}^{(2)} _1$ and consequently, it is a biholomorphism. This completes the proof.
\end{proof}
We are again at the stage where one should ask for a symmetrization map equivalent to the one given by equation (\ref{Definition of sym}). From the proof of previous theorem, it is clear that whenever $(s:t:u:v)\in \mathcal{D}^{(2)} _1$, we must have $t\neq 0$. Hence, the map $sym_{\mathcal{D}^{(2)}_1}: \mathcal{D}^{(2)} _1\rightarrow\C^2$ sending $(s:t:u:v)$ to $\Big(\frac{u}{t},\frac{v}{t}\Big) $ is well defined and holomorphic. It is indeed the map equivalent to $sym$. Let us state a result similar to Theorem \ref{Equivalent sym map for Omega_1}. We leave the details for the reader.
\begin{theorem}
	There is a proper holomorphic map $sym_{\mathcal{D}^{(2)} _1} : \mathcal{D}^{(2)} _1 \rightarrow \mathcal{D}_1$ (defined above) such that the following diagram is commutative
	\[ \begin{tikzcd}
	\D \times \D \arrow{r}{J} \arrow[swap]{d}{sym} & \mathcal{D}^{(2)} _1 \arrow{d}{sym_{{\mathcal{D}^{(2)} _1}}} \\%
	\G \arrow{r}{F}& \mathcal{D}_1
	\end{tikzcd},
	\]
	where $F$ and $J$ are defined in Theorem \ref{Biholom B/w spaces} and \ref{Biholo map between D^2 and {D}^{(2)} _1}, repectively.
\end{theorem}

\vspace{0.1in} \noindent\textbf{Acknowledgement:}
This research is supported by the University Grants Commission Centre for Advanced Studies. The first author thanks Prof.  K. Verma and the second author thanks Prof. B. Datta for fruitful discussions.


\begin{thebibliography}{99}
	
	\bibitem{Agler Lykova Young} J. Agler, Z Lykova, N. Young, \textit{A geometric characterization of the symmetrized bidisc}, J. Math. Anal. Appl. 473 (2019), no. 2, $1377-1413$.
	
	\bibitem{ay-jfa} J. Agler and N. J. Young, \textit{A commutant lifting theorem for a domain in $C^2$ and spectral interpolation}, J. Funct. Anal. 161 (1999), no. 2, $452 - 477$.
	
	\bibitem{A-Y} J. Agler and N. J. Young, \textit{The hyperbolic geometry of the symmetrized bidisc,} J. Geom. Anal. 14 (2004), pp. $375-403$.
	
	
	\bibitem{BD} E. Bedford and J. Dadok, \textit{Bounded domains with prescribed group of automorphisms}, Comment. Math. Helv. 62 (1987), no. 4, 561–572.
	
	\bibitem{BPSR} T. Bhattacharyya, S. Pal and S. Shyam Roy, \textit{Dilations of $\Gamma$-contractions by solving operator equations}, Adv. in Math. 230 (2012), 577-606.
	
	\bibitem{Camacho} C. Camacho, A. Lins Neto, \textit{Geometric Theory of Foliations,} Birkhäuser Boston, Inc., Boston, MA, 1985, translated from the Portuguese by Sue E. Goodman.
	
	\bibitem{Cartan} E Cartan, \textit{Sur la g\'eom\'etrie pseudo-conforme des hypersurfaces de l'espace de deux variables complexes I}, Ann. Mat. Pura Appl. 11 (1933) 17-90 (or Oeuvres Completes II, 2, 1231-1304)
	
	
	\bibitem{Guillemin} V. Guillemin and A. Pollack, \textit{Differential Topology,} Prentice-Hall, Englewood Cliffs, 1974.
	
	\bibitem{Isaev_n_n2-1} A. V. Isaev, \textit{
		Hyperbolic Manifolds of Dimension $n$ with Automorphism Group of Dimension $n^2- 1$
	}, The Journal of Geometric Analysis, Volume 15, Number 2, (2005), $239-259$.
	
	\bibitem{Isaev} A. V. Isaev, \textit{Hyperbolic 2-dimensional manifolds with 3-dimensional automorphism group}, Geometry and Topology 12 (2008), no. 2, $643-711$.
	
	\bibitem{Turgay} H. Turgay Kaptano\u glu,
	\textit{Möbius-invariant spaces and algebras in polydiscs}.
	Indiana Univ. Math. J. 41 (1992), no. 2, $339-362$.
	
	
	\bibitem{J-F-Invariant} M. Jarnicki, P. Pflug, \textit{Invariant Distances and Metrics in Complex Analysis,} 2nd extended edition, De Gruyter, Berlin, 2013.
	
	
	\bibitem{Kaniuth} E. Kaniuth, \textit{A course in commutative Banach algebras,}
	Graduate Texts in Mathematics, 246. Springer, New York, 2009.
	
	\bibitem{KZ} {\L}. Kosi\'nski and W. Zwonek, \textit{Nevanlinna-Pick problem and uniqueness of left inverses in convex domains, symmetrized bidisc and tetrablock}, J. Geom. Anal. 26 (2016), no. 3, $1863–1890$.
	
	\bibitem{KrantzSurvey} S. G. Krantz, \textit{The automorphism groups of domains in complex space: a survey},
	Quaest. Math. 36 (2013), no. 2, $225-251$.
	
	\bibitem{Lee Smooth} J. M. Lee, \textit{Introduction to smooth manifolds}, Second edition. Graduate Texts in Mathematics, 218. Springer, New York, 2013.
	
	\bibitem{Rudin Nagel} Nagel and W. Rudin, \textit{Moebius-invariant function spaces on balls and spheres},
	Duke Math. J. 43 (1976), no. 4, $841-865$.
	
	
	\bibitem{Ohsawa} T. Ohsawa,
	\textit{Analysis of several complex variables.}
	Translated from the Japanese by Shu Gilbert Nakamura. Translations of Mathematical Monographs, 211. Iwanami Series in Modern Mathematics. American Mathematical Society, Providence, RI, 2002.
	
	\bibitem{Rudin poly} W. Rudin, \textit{Function theory in polydiscs}, W. A. Benjamin, Inc., New York-Amsterdam 1969
	
	
	\bibitem{SZ} R. Saerens and W. R. Zame, \textit{The isometry groups of manifolds and the automorphism groups of domains}, Trans. Amer. Math. Soc. 301 (1987), no. 1, $413-429$.
\end{thebibliography}
\end{document}